\documentclass{amsart}
\usepackage{graphicx,color}
\usepackage{latexsym}
\usepackage{amsfonts}
\usepackage[all]{xy}
\usepackage{amssymb, mathrsfs, amsfonts, amsmath}
\usepackage{amsbsy}
\usepackage{amsfonts}
\newtheorem*{acknowledgement}{Acknowledgement}

\newtheorem{lemma}{Lemma}
\newtheorem{proposition}{Proposition}
\newtheorem{remark}{Remark}
\newtheorem{theorem}{Theorem}

\numberwithin{equation}{section}

\newcommand{\minvol}{{\rm Min\,Vol}}
\newcommand{\vol}{{\rm Vol}}

\title[Four-dimensional manifolds]{Minimal volume invariants, topological sphere theorems and biorthogonal curvature on 4-manifolds}
\author{E. Costa}
\author{E. Ribeiro Jr.}
\address[E. Costa]{Universidade Federal da Bahia - UFBA, Departamento de Matem\'{a}tica, Campus de Ondina, Av. Ademar de Barros, 40170-110, Salvador / BA, Brazil.}\email{ezio@ufba.br}
\address[E. Ribeiro Jr]{Universidade Federal do Cear\'a - UFC, Departamento  de Matem\'atica, Campus do Pici, Av. Humberto Monte, Bloco 914,
60455-760, Fortaleza / CE, Brazil.}\email{ernani@mat.ufc.br}

\allowdisplaybreaks
\numberwithin{equation}{section}
\numberwithin{theorem}{section}

\thanks{E. Ribeiro Jr was partially supported by grants from CNPq/Brazil (Grant: 303091/2015-0) and PRONEX-FUNCAP/CNPq/Brazi}

\begin{document}

\begin{abstract}
The goal of this article is to establish estimates involving the Yamabe minimal volume, mixed minimal volume and some topological invariants on compact 4-ma\-ni\-folds. In addition, we provide topological sphere theorems for compact subma\-ni\-folds of spheres and Euclidean spaces, provided that the full norm of the second fundamental form is suitably bounded.
\end{abstract}

\maketitle
\newcommand{\spacing}[1]{\renewcommand{\baselinestretch}{#1}\large\normalsize}
\spacing{1.2}

\section{Introduction}\label{introduction}

\subsection{Minimal Volume Invariants}
\label{minvolinvar}

Let $M^n$ be an $n$-dimen\-sio\-nal compact oriented smooth manifold with scalar curvature $s_g,$ or simply $s,$ sectional curvature $K$ and $\mathcal{M}$ be the set of smooth Riemannian structures on $M^n.$ Moreover, we consider all complete Riemannian structures $g\in\mathcal{M}$ whose sectional curvatures satisfy $|K(g)|\leq 1.$ With these settings, Gromov \cite{Gromov}, in his seminal paper on bounded cohomology, introduced the concept of minimal volume. More precisely, the {\it minimal volume} of $M^n$ is defined by
\begin{equation}
\label{minvol}
\minvol(M)=\inf_{|K(g)|\leq1}\vol(M,g).
\end{equation}
This concept plays an important role in geometric topology. It is closely related with other important invariants as, for instance, {\it minimal entropy} ${\rm h(M)}$ and {\it simplicial volume} $\| M\|.$ Paternain and Petean \cite{PP} proved that the minimal volume, on a compact manifold $M^n,$ satisfies the following chain of inequalities
\begin{equation}
c(n)\| M\|\leq [{\rm h(M)}]^n \leq (n-1)^{n} {\rm Min\,Vol(M)},
\end{equation}where $c(n)$ is a positive constant; for more details, we refer the reader to \cite{Kotschick,LS} and \cite{PP}. It should be emphasized that some authors have studied other minimal volume invariants in a similar context (cf. \cite{LeBrun,Sung} and \cite{SungAnnals}). Among them, let us highlight the following ones: {\it Gromov minimal volume}, which is defined by
\begin{equation}
\label{Gromovminvol}
\vol_K(M)=\inf\{\vol(M,g);\,K_g\geq-1\},
\end{equation} as well as the {\it Yamabe minimal volume} which is given by
\begin{equation}
\label{Yamabeminvol}
\vol_s(M)=\inf\left\{\vol(M,g);\,\frac{s_g}{n(n-1)}\geq-1\right\}.
\end{equation} The Yamabe minimal volume measures how much the negative scalar curvature is inevitable on a compact manifold. In fact, if a compact manifold $M^n$ admits a metric with nonnegative scalar curvature, then $\vol_s(M)=0.$ In general, the computation of these invariants is not easy. However, Petean \cite{Petean2} showed brightly that any compact simply connected manifold $M^n$ of dimension $n\ge 5$ has $\vol_s(M)=0.$ In other words, it collapses with scalar curvature bounded from below. While LeBrun  \cite{LeBrun} defined a new type of minimal volume called {\it mixed minimal volume}, which is given by
\begin{equation}
\label{mixedminvol}
\vol_{K,s}(M)=\inf\left\{\vol(M,g);\,\frac{1}{2}\left(K_g+\frac{s_g}{n(n-1)}\right)\geq-1\right\}.
\end{equation} Another interesting mixed minimal volumes were studied by Sung (cf. \cite{Sung} and \cite{SungAnnals}).

Before proceeding, we recall that a compact complex surface $M^4$ is said of {\it general type} if the Kodaira dimension of $M^4$ is equal to 2; for details see \cite{LeBrunCAG}. Lebrun  \cite{LeBrunTrends} showed that  a compact complex surface $M^4$ with the first Betti number $b_{1}(M)$ even is of general type if and only if its Yamabe invariant $Y(M)$ is negative. The hypothesis that $b_{1}(M)$ is even is equivalent to requiring that the complex surface $M^4$ admits a K\"ahler metric.

The Seiberg-Witten equations are
\begin{eqnarray*}
\left \{ \begin{array}{ll}
D^{A}\phi=0 \\
F_{A}^{+}=\sigma(\phi),
                    \end{array} \right.
\end{eqnarray*} where $A$ is a Hermitian connection on the line bundle $\mathcal{L},$ $D^{A}$ denotes the Dirac operator induced by $A,$ while $F_{A}^{+}$ is the self-dual part of the imaginary-valued curvature $2$-form of $A$ and $\sigma$ is a certain canonical real-quadratic map (cf. \cite{scorpan} and \cite{witten}). Many of the most remarkable consequences of Seiberg-Witten theory stem from the fact that a solution $\phi$ of the Seiberg-Witten equations satisfies the Weitzenb\"och formula 
\begin{equation}
2\Delta \mid \phi\mid^{2}+4\mid \nabla \phi \mid^{2}+s\mid \phi\mid^{2}+\mid \phi\mid^{4}=0.
\end{equation} It follows from Seiberg-Witten theory that any two diffeomorphic complex algebraic surfaces must have the same Kodaira dimension. Also, by using the Seiberg-Witten theory it is possible to compute the Yamabe invariants of most of complex algebraic surfaces. Nonetheless, the Yamabe invariants often distinguish different smooth structures on the same topological 4-manifold. Generally speaking, in dimension 4, the Yamabe invariant alone is weak to control the topology of a given manifold. For this reason, it is expected one additional condition. We refer the reader to Chapter 10 in\cite{scorpan} for a detailed discussion on Seiberg-Witten theory; see also \cite{Donaldson,LeBrun,LeBrunP} and \cite{witten}.

A complex surface is said to be {\it minimal} if it is not the blow-up of some other complex surface. Any compact complex surface $M^4$ can be obtained from some minimal complex surface $X,$ called {\it minimal model} for $M^4,$ by blowing up $X$ at a finite number of points. In particular, from Kodaira's classification theory a complex surface is of general type if and only if it has a minimal model with $c_{1}^{2}>0$ and $c_{1}\cdot [\omega]<0$ for some K\"ahler class, where $c_{1}$ is the first Chern class. The compact complex-hyperbolic 4-manifold  $\Bbb{C}\mathcal{H}^{2}/\Gamma$ is a classical example of minimal compact complex surface of general type. Usually, the complex-hyperbolic plane $\Bbb{C}\mathcal{H}^{2}$ can be seen as the unit ball in $\Bbb{C}^{2}$ endowed with the Bergmann metric. Furthermore, it is well-known that $c_{1}^{2}(X)=\big(2\chi(X)+3\tau(X)\big),$ where $\chi$ and $\tau$ stand, respectively, for the Euler cha\-rac\-teristic and the signature. For more details, see, for instance \cite{Barth} and \cite{LeBrunAnn}.

In 2011, LeBrun  \cite{LeBrun} showed the following result via Seiberg-Witten theory.

\begin{theorem}[LeBrun, \cite{LeBrun}]
\label{thLeBrun}
Let $M^4=X\sharp j \Bbb{CP}^2$ be a compact K\"ahler surface, where $X$ is the minimal model of $M^4.$ If the Yamabe invariant of $M^4$ is negative, then
$$\vol_{K,s}(M^4)\geq \frac{9}{4}\vol_s(M^4).$$
Moreover, the equality holds if $M^4$ is a compact complex-hyperbolic 4-manifold $\Bbb{C}\mathcal{H}^{2}/\Gamma.$
\end{theorem}

It is worth mentioning that $\vol_s(M)=\frac{2\pi^{2}}{9}c_{1}^{2}(X),$ where $X$ is the minimal model for $M^4$ (cf. \cite{LeBrunMRS}).

The primary goal of this article is to provide some estimates involving Yamabe minimal volume, mixed minimal volume and some topological invariants on compact 4-ma\-ni\-folds. In addition, we shall provide topological sphere theorems for compact submanifolds of spheres and Euclidean spaces provided that the full norm of the second fundamental form is bounded. In order to do so, let us remember that, for each plane $P\subset T_{x}M$ at a point $x\in M^4,$ the {\it biorthogonal (sectional) curvature} of $P$ is given by the following average of the sectional curvatures
\begin{equation}\label{[1.2]}
 \displaystyle{K^\perp (P) = \frac{K(P) + K(P^\perp) }{2}},
\end{equation}
where $P^\perp$  is the orthogonal plane to  $P.$ In particular,  for each point $x\in M^4,$ we take the minimum of bior\-tho\-go\-nal curvature to obtain the following function
\begin{equation}
\label{[1.3]}
K_1^\perp(x) = \textmd{min} \{K^\perp(P); P  \textmd{ is a 2- plane in } T_{x}M \}.
\end{equation} The sum of pair of sectional curvatures on two orthogonal planes, which was perhaps first observed by Chern \cite{Chern}, plays a crucial role in dimension four. Surprisingly, the positivity of the biorthogonal curvature
is an intermediate condition between positive sectional curvature and positive
scalar curvature. Thereby, it is expected to get interesting results in considering this approach. Gray \cite{Gray} showed that the Euler characteristic $\chi(M)$ of a compact oriented 4-manifold $M^4$ is nonnegative, provided that the sectional curvature satisfies $\frac{K(P^\perp)}{ K(P)} \geq   \frac{3}{4},$ whenever $K(P)\neq 0.$ Singer and Thorpe \cite{ST} observed that a $4$-manifold $(M^4,\,g)$ is Einstein if and only if $K^\perp(P)=K(P),$ for any plane $P\subset T_{x}M$ at any point $x\in M^4.$ From Seaman \cite{SeamanTAMS} and Costa and Ribeiro \cite{CR}, $\Bbb{S}^4$ and $\Bbb{CP}^2$ are the only compact simply-connected 4-manifolds with positive biorthogonal curvature that can have (weakly) $1/4$-pinched biorthogonal curvature, or nonnegative isotropic curvature, or satisfying $K^{\perp} \ge \frac{s}{24} > 0.$ In \cite{renato2}, Bettiol has proven that the positivity of biorthogonal curvature is preserved under connected sums. Furthermore, he showed that $\Bbb{S}^4,$ $\sharp^{m}\Bbb{CP}^2\sharp^{n}\overline{\Bbb{CP}}^2$ and $\sharp^{n}\big(\Bbb{S}^2\times \Bbb{S}^2\big)$ admit metrics with positive biorthogonal curvature. For more details on this subject, see, for instance \cite{renato,renato2,CR,Ribeiro,Noronha,noronha2} and \cite{Seaman}.

In order to proceed, we recall that an extension of the Yamabe problem was started by Gursky and LeBrun (cf. \cite{Gursky1} and \cite{GLB}). To do so, they used a modified scalar curvature which obeys a conformal invariance property. Afterward, based on ideas developed by Lebrun and Gursky \cite{GLB}, Cheng and Zhu \cite{CZ} studied the modified Yamabe problem in terms of a functional depending on Weyl curvature tensor (see also \cite{listing10}, Section 2.2), which is a fundamental tool in this work. Itoh \cite{Itoh} described the global geometry of the modified scalar curvature.  For the sake of completeness let us sketch here this construction. Initially, consider maps $f:\mathcal{W}_{M}\to C^{0,\alpha}(M)$ satisfying 
\begin{eqnarray}
\label{F}
\left \{ \begin{array}{lll}
       f(W_{g})\geq 0,\\
f(W_{\overline{g}})=u^{-2}f(W_{g}),\\
\bar{g}=u^2g,                
\end{array} \right.
\end{eqnarray} where $\mathcal{W}_{M}$ denotes the space consisting of conformal curvature tensors $W_{g},$ $g\in\mathcal{M}.$ Hence, the {\it modified Yamabe functional}  is given by
\begin{equation}
\label{yamabemodified}
  \displaystyle{Y^{f}(M^4,\,g)=\frac{1}{\vol(M,g)^{\frac{1}{2}}}\int_{M}\big(s_{g} -f(W)\big)dV_g},
\end{equation} where, in this case, $s-f(W)$ is called {\it modified scalar curvature} of $(M^4,g).$ Denoting by $[g]$ the conformal class of $g,$ we deduce that $Y^{f}(M^4, [g]) =  \inf_{\overline{g}\in [g]} Y^{f}(M^{4}, \overline{g})$ and $Y^{f}(M) = \sup_{g\in\mathcal{M}} Y^{f}(M^{4}, [g])$ are the {\it modified Yamabe constant} and the {\it modified Yamabe invariant}, respectively (cf. \cite{CZ} and \cite{Itoh}).

In \cite{CDR}, Costa et al. were able to show that $12K_{1}^{\perp}$ is a modified scalar curvature as well as $Y_1^\perp(M^4, [g]),$ defined by
\begin{equation}
Y_1^\perp(M^4, [g]) =  \inf_{\overline{g} \in [g]}\, \frac{12}{\vol(M,g)^{\frac{1}{2}}}\int_{M}\overline{K}_1^\perp dV_{\overline{g}},
\end{equation} is a {\it modified Yamabe constant}, where $[g]$ denotes the conformal class of $g.$ Thus, we conclude that
\begin{equation}
\label{Yi}Y_{1}^{\perp}(M)=\sup_{g\in\mathcal{M}} Y_{1}^{\perp}(M,[g])
\end{equation} is its {\it modified Yamabe invariant}. In particular, we have $Y_1^{\perp}(M)\leq Y(M),$ where $Y(M)$ denotes the standard Yamabe invariant of $M^4.$ We further notice that:

\begin{eqnarray*}
\left \{ \begin{array}{lll}
Y_1^{\perp}(\Bbb{C}\mathcal{H}^{2}/\Gamma) \geq \frac{1}{6}Y(\Bbb{C}\mathcal{H}^{2}/\Gamma), \\
 Y_{1}^{\perp}(\Bbb{H}^{2}\times\Bbb{H}^{2}/\Gamma) \geq \frac{1}{4} Y(\Bbb{H}^{2}\times\Bbb{H}^{2}/\Gamma)\,\,\,\,\,\,\,\hbox{and}\\
Y_1^{\perp}(\Bbb{H}^4/\Gamma)=Y(\Bbb{H}^4/\Gamma)=-8\pi\sqrt{3\chi(\Bbb{H}^4/\Gamma)}.
                    \end{array} \right.
\end{eqnarray*}

In this article, mainly motivated by outstanding ideas outlined by Gursky \cite{Gursky1} and LeBrun \cite{LeBrun}, we use the notion of biorthogonal (sectional) curvature (see also Eq. (\ref{[1.6]}) in Section \ref{prel}) to define a minimal volume invariant of $M^4$ by setting
\begin{equation}
\label{mixedminvolmod}
\vol_{K_{1}^{\perp} , s}(M)=\inf\left\{\vol(M,g);\,\frac{1}{2}\left(K_1^\perp+\frac{s}{12}\right)\geq-1\right\}.
\end{equation} The key ingredient here that should be emphasized is that $\frac{1}{2}\big(K_1^\perp+\frac{s}{12}\big)$ is a modified scalar curvature (cf. \cite{CDR} and \cite{Itoh}). This fact plays a crucial role in the proof of our first result. For what follows, we set $Y_{K_{1}^{\perp},s}(M)$ to be its corresponding modified Yamabe invariant.

After these preliminary remarks we may state our first result, which can be compared with Theorem \ref{thLeBrun} by LeBrun.

\begin{theorem}\label{thmT}
Let $M^4=X\sharp j \Bbb{CP}^2$ be a compact K\"ahler surface, where $X$ is the minimal model of $M^4.$ Suppose that $Y(M)<0.$ Then we have:
\begin{equation}
\vol_{K,s}(M)\geq\vol_{K_1^\perp,s}(M)\geq|Y_{K_{1}^{\perp},s}(M)|^{2}\geq\frac{9}{4}\vol_s(M).
\end{equation} Moreover, the equalities hold if $M^4$ is the compact complex-hyperbolic 4-manifold $\Bbb{C}\mathcal{H}^{2}/\Gamma.$
\end{theorem}

\begin{remark}
As it was previously mentioned a four-dimensional Riemannian manifold $(M^4,\,g)$ is Einstein if and only if $K^\perp=K.$ Therefore, it is not difficult to check that $$\vol_{K,s}(M)=\vol_{K_1^\perp,s}(M)=\frac{9}{4}\vol_s(M),$$ for any compact complex-hyperbolic 4-manifold $\Bbb{C}\mathcal{H}^{2}/\Gamma.$ Furthermore, we claim that
\begin{equation}
\label{eq679}
\vol_{K_{1}^{\perp},s}(\Bbb{H}^2\times\Bbb{H}^2 /\Gamma)\leq \frac{8\pi^{2}}{3}\chi(\Bbb{H}^2\times\Bbb{H}^2 /\Gamma).
\end{equation} Indeed, we consider $\Bbb{H}^2\times\Bbb{H}^2 /\Gamma$ endowed with its canonical K\"ahler-Einstein metric. In this case, we have $K_{1}^{\perp}=\frac{s}{4}$ and $|W^{+}|^{2}=|W^{-}|^{2}=\frac{s^{2}}{24}.$ Therefore, by using Chern-Gauss-Bonnet formula (\ref{characteristic}) we infer $$96\pi^{2}\chi\big(\Bbb{H}^2\times\Bbb{H}^2 /\Gamma\big)=s^{2}\vol\big(\Bbb{H}^2\times\Bbb{H}^2 /\Gamma\big).$$ Choosing the normalized scalar curvature $s=-6,$ we immediately obtain $\frac{1}{2}\Big(s+12K_{1}^{\perp}\Big)=-12,$ which settles our claim.
\end{remark}

In the work \cite{LeBrunP}, LeBrun considered $\mathcal{C}\subset H^{2}(M,\Bbb{R})$ to be the set of {\it monopole classes}, which are the first Chern classes of those $spin^{c}$ structures on $M^4$ for which the Seiberg-Witten equation have solution for all metrics. In particular, its convex hull, denoted by ${\rm {\bf Hull}}(\mathcal{C}),$ is compact. From this, he introduced a real-valued invariant of $M^4$ by setting

\begin{equation}
\beta^{2}(M)=\max \left\{\int_{M}\alpha \wedge \alpha ;\,\, \alpha\in {\rm {\bf Hull}}(\mathcal{C})\right\}
\end{equation} if $\mathcal{C}$ is not empty. Otherwise, we consider $\beta^{2}(M)=0.$ There are many $4$-manifolds $M^4$ for which $\beta^{2}(M)>0.$ With these definitions, LeBrun showed that a compact oriented $4$-manifold with $b_{2}^{+}(M)\geq 2$ satisfies
\begin{equation}
\int_{M}s^2 dV_{g}\geq 32\pi^{2}\beta^{2}(M).
\end{equation} In addition, if $\mathcal{C}$ is not empty, we have
\begin{equation}
Y(M)\leq -4\pi\sqrt{2\beta^{2}(M)}.
\end{equation} We refer to \cite{Donaldson,KotschickM,{LeBrunP}} for a general discussion on monopole classes.

Next, we shall show a relation between the LeBrun's invariant $\beta^{2}(M)$ and the modified scalar curvature $\frac{1}{2}\big(K_1^\perp+\frac{s}{12}\big)$ as well as its corresponding modified Yamabe invariant $Y_{K_{1}^{\perp},s}(M).$ To be precise, we have established the following result.

\begin{theorem}
\label{thmmono}
Let $M^4$ be a 4-dimensional compact oriented manifold. Suppose that $\mathcal{C}$ is not empty and $b_{2}^{+}(M)\geq 2.$ Then, for any metric $g$ on $M^4,$ we have:
\begin{equation}
\int_{M}\Big[\frac{1}{2}\big(K_{1}^{\perp}+\frac{s}{12}\big)\Big]^{2}dV_{g}\ge \frac{\pi^2}{4}\beta^{2}(M).
\end{equation} In particular, the equality is attained by $\Bbb{C}\mathcal{H}^{2}/\Gamma.$ In addition, if $Y(M)<0,$ then
\begin{equation}
|Y_{K_{1}^{\perp},s}(M)|^{2}\ge 36\pi^2\beta^{2}(M).
\end{equation}
\end{theorem}

\subsection{Sphere Theorem for Submanifolds}

It is very interesting to investigate curvature and topology of submanifolds of spheres and Euclidean spaces. In this context, in 1973, Lawson and Simons \cite{LSi}, by means of nonexistence for stable currents on compact submanifolds of a sphere, obtained a criterion for the vanishing of the homology groups of compact submanifolds of spheres. In particular, they showed the following result.

\begin{theorem}[Lawson-Simons, \cite{LSi}]
\label{thmLSi} Let $M^n$ be an $n(\ge 4)$-dimensional oriented compact submanifold in the unit sphere $\Bbb{S}^{n+p}.$
\begin{enumerate}
\item If $n=4,$ and the second fundamental form $\alpha$ of $M^4$ satisfies $\| \alpha \|^{2}<3,$ then $M^4$ is a homotopic sphere.
\item If $n\geq 5,$ and the second fundamental form $\alpha$ of $M^n$ satisfies $\| \alpha \|^{2}<2\sqrt{n-1},$ then $M^n$ is homeomorphic to a sphere.
\end{enumerate}
\end{theorem} Afterward, Leung \cite{Leung} and Xin \cite{Xin} were able to extend the results obtained by Lawson and Simons to compact submanifolds of Euclidean spaces. In 2001, inspired by ideas developed in  \cite{LSi,Leung} and \cite{Xin}, Asperti and Costa \cite{AC} obtained an estimate for the Ricci curvature of submanifolds of a space form which improves Leung's estimates. As a consequence, Asperti and Costa obtained a new criterion for the vanishing of the homology groups of compact submanifolds of spheres and Euclidean spaces. In 2009, Xu and Zhao \cite{XZ} investigated the topological and differentiable structures of submanifolds by imposing certain conditions on the second fundamental form. In the work \cite{GX}, Gu and Xu used the convergence results obtained by Hamilton and Brendle as well as Lawson-Simons-Xin formulae to obtain a differentiable sphere theorem for submanifolds in space forms. Similar result was obtained by Andrews-Baker \cite{Andrews} making use of the mean curvature flow. For  comprehensive references on such a subject, we indicate, for instance \cite{AC,GX,LSi,Leung} and \cite{XZ}.

Before to state our next results let us fix notation. We shall denote by $f: M^{n}\to \mathcal{Q}_{c}^{n+m}$ an isometric immersion of a connected n-dimensional compact Riemannian manifold $M^n$ into a complete, simply connected $(n+m)$-dimensional manifold $\mathcal{Q}_{c}^{n+m}$ with constant sectional curvature $c.$ We also denote by $T_{p}M$ the tangent space of $M^n$, at each point $p\in M^n,$ and $\big(T_{p}M\big)^{\perp}$ stands for the normal space of the immersion $f$ at  $p.$ Furthermore, $\vec{H}$ and $\alpha: T_{p}M\times T_{p}M\to \big(T_{p}M\big)^{\perp}$ stand for the mean curvature vector and the second fundamental form of the immersion, respectively. We adopt the following convention:
\vspace{0.3cm}

{\it If $p\in M^n$ is such that $\vec{H}(p)\neq 0,$ then $\lambda_{1}\leq ...\leq \lambda_{n}$ denote the eigenvalues of the Weingarten operator $A_{\xi_{1}}$ with $\xi_{1}=\frac{1}{H}\vec{H}(p),$ where $H=\|\vec{H}\|$ is the length of the mean curvature vector. Otherwise, if $\vec{H}(p)= 0,$ we then pick  $\lambda_{i}=0$ for $1\leq i\leq n,$ and $\xi_{1}$ any unit vector normal to $M^n$ at the point $p.$}

\vspace{0.3cm}

With this convention, we recall a result by Asperti and Costa \cite{AC}, which is crucial for our purposes.

\begin{theorem}[Asperti-Costa, \cite{AC}]
\label{thmAC}
Let $f: M^{n}\to \mathcal{Q}_{c}^{n+m}$ be an isometric immersion, where $M^n$ is a compact oriented manifold and $c\geq 0.$ Suppose that
\begin{equation}
\label{estAC}
\|\alpha\|^{2}<\frac{n^{2}H^{2}}{(n-p)}+\frac{n(n-2p)H\lambda_{1}}{(n-p)}+nc
\end{equation} holds on $M^n$ for some integer $p$ satisfying $2\leq p\leq \frac{n}{2}.$ Then the k-th homology group $H_{k}(M,\Bbb{Z})=0,$ for $p\leq k\leq n-p.$
\end{theorem}

Here, motivated by works \cite{Andrews,AC} and \cite{LSi}, we going to provide topological sphere theorems for compact submanifolds of spheres and Euclidean spaces, provided that the full norm of the second fundamental form is bounded by a multiple of the length of the mean curvature vector. More precisely, we may announce our next result as follows.

\begin{theorem}
\label{thmSub}
Let $M^4$ be a connected 4-dimensional oriented compact submanifold of $\mathcal{Q}_{c}^{4+m}$ with $c\geq 0.$ Then we have:

\begin{equation}
4K_{1}^{\perp}\ge -\|\alpha\|^{2}+4\big(2H^{2}+c\big).
\end{equation} In particular, if $M^4$ has finite fundamental group and $\|\alpha\|^{2}<4\big(2H^{2}+c\big),$ then $M^4$ is homeomorphic to a sphere $\Bbb{S}^4.$
\end{theorem}

The estimate obtained in Theorem \ref{thmSub} improves the estimate stated in Theorem 4 in \cite{GX}. Moreover, it can be seen as a generalization, in the topological sense, of the same theorem. Notice that the condition $\|\alpha\|^{2}<\frac{n^{2}H^{2}}{n-1}+2c$ used in Theorem 4 in \cite{GX} implies that $M^4$ has nonnegative sectional curvature.

In the sequel, as an application of Theorems \ref{thmAC} and \ref{thmSub} we have the following result.

\begin{theorem}
\label{thmSub2}
Let $M^4$ be a $4$-dimensional oriented compact submanifold in the unit sphere $\Bbb{S}^{4+m}$, $m\ge 1.$

\begin{enumerate}
\item If $\|\alpha\|^{2}<4,$ then $M^4$ has positive biorthogonal curvature. In addition, if $M^4$ has finite fundamental group, then $M^4$ is homeomorphic to a sphere $\Bbb{S}^4.$
\item If $\|\alpha\|^{2}\leq 4,$ then $M^4$ has nonnegative biorthogonal curvature. Moreover, if for every point of $M^4$ some biorthogonal curvature vanishes, then $M^4$ is a minimal submanifold of $\Bbb{S}^{n}$. In addition, if $m=1$, then $M^4$ must be $\Bbb{S}_{c_{1}}^{2}\times \Bbb{S}_{c_{2}}^{2}.$
\end{enumerate}
\end{theorem}

This last theorem can be seen as an improvement of Theorem \ref{thmLSi} in dimension four. Moreover, the result obtained in the second item of Theorem \ref{thmSub2} generalizes the main result in \cite{Leung2}.

\section{Background}
\label{prel}
Throughout this section we review some basic facts and notation that will be useful for the establishment of the desired results.

\subsection{Four-Manifolds} As it was previously remarked 4-manifolds display fascinating and peculiar features. Indeed, many peculiar features are directly attributable the fact that the bundle of $2$-forms on a four-dimensional oriented Riemannian manifold can be invariantly decomposed as a direct sum
\begin{equation}
\label{piu}
\Lambda^2=\Lambda^{+}\oplus\Lambda^{-},
\end{equation}  where $\Lambda^{\pm}$ is the $(\pm 1)$-eigenspace of the Hodge star operator. The decomposition (\ref{piu}) is conformally invariant. Moreover, it allows us to conclude that the Weyl tensor $W$ is an endomorphism of $\Lambda^2=\Lambda^{+} \oplus \Lambda^{-}$ such that $W = W^+\oplus W^-.$  A manifold is {\it locally conformally flat} if $W=0.$ It is said {\it half-conformally flat} if either $W^{-}=0$ or $W^{+}=0.$ Furthermore, an oriented manifold is {\it self-dual} if $W^{-}=0.$ We point out that the complex projective space $\Bbb{CP}^2$ endowed with Fubini-Study metric shows that,  in real dimension 4, the half-conformally flat condition is really weaker than locally conformally flat condition. Moreover, since the Riemann curvature tensor $\mathcal{R}$ of $M^4$ can be seen as a linear map on $\Lambda^2,$ we have the following decomposition

\begin{equation}
\mathcal{R}=
\left(
  \begin{array}{c|c}
    \\
W^{+} +\frac{s}{12}Id & \mathring{Ric} \\ [0.4cm]\hline\\

    \mathring{Ric}^{*} & W^{-}+\frac{s}{12}Id  \\[0.4cm]
  \end{array}
\right),
\end{equation} where $\mathring{Ric}$ is the traceless Ricci tensor.

Let $H^\pm(M^4,\Bbb{R})$ be the space of positive and negative harmonic 2-forms, respectively. Therefore, the second Betti number $b_2$ of $M^4$ can be written as $b_2 = b_{2}^{+} + b_{2}^{-}$,  where $b_{2}^\pm$ = dim $H^\pm(M^4,\Bbb{R}).$ In particular, the signature of $M^4$ is given by $$\tau(M)=b_{2}^{+} - b_{2}^{-}.$$ By the Hirzebruch signature theorem, it can be expressed as
\begin{equation}
\label{signature} 12\pi^2\tau(M) = \int_{M} \Big(| W^+|^2 - |W^-|^2\Big) dV_{g}.
\end{equation} In addition, by Chern-Gauss-Bonnet formula, the Euler characteristic of $M^4$ can be written as
\begin{equation}
\label{characteristic} 8\pi^2 \chi(M) = \int_{M} \Big(| W^+ |^2 + |W^-|^2 + \frac{s^2}{24} - \frac{1}{2}|\mathring{Ric}|^2\Big) dV_{g}.
\end{equation}

We now  fix a point and diagonalize $W^\pm$ such that $w_i^\pm,$ $1\le i \le 3,$ are their respective eigenvalues. In particular, we remark that the eigenvalues of  $W^\pm$ satisfy
\begin{equation}
\label{eigenvalues}
w_1^{\pm}\leq w_2^{\pm}\leq w_3^{\pm}\,\,\,\,\hbox{and}\,\,\,\,w_1^{\pm}+w_2^{\pm}+w_3^{\pm} = 0.
\end{equation} Easily one verifies from (\ref{eigenvalues}) that
\begin{equation}
\label{estcos}
|W^\pm|^2\leq6(w_1^\pm)^2.
\end{equation}

Next, as it was pointed out in \cite{CR} and \cite{Seaman}, Eq. (\ref{[1.3]}) provides the following useful identity
\begin{equation}
\label{[1.6]}
K_1^\perp = \frac{w_1^+ + w_1^-}{2}+\frac{s}{12}.
\end{equation} For more details see \cite{CR}. 

We also remember that a differential $2$-form $\omega$ on a manifold $M$ gives at each point $p\in M$ a bilinear form on the tangent space $T_{p}M$ given by $\omega:T_{p}M\times T_{p}M\to \Bbb{R}.$ In particular, we say that $\omega$ is non-degenerate if $\omega_{p}$ is non-degenerate for all $p\in M,$ namely, $\omega_{p}(v,w)=0$ for all $w\in T_{p}M$ implies $v=0.$ 

Now, we shall use the classical Weitzenb\"och formula to obtain our first lemma.

\begin{lemma}
\label{lemM}
Let $\alpha^{+}\in H^{+}(M^{4},\Bbb{R})$ be a non-degenerate harmonic 2-form on a compact oriented $4$-manifold $M^4.$ Then we have:
$$\int_{M}|\nabla \alpha^{+}|^{2}dV_{g}\geq \frac{2}{3}\int_{M}\big(6K_{1}^{\perp}-s\big)|\alpha^{+}|^{2}dV_{g}.$$ In particular, the equality holds if and only if $W^{-}=0$ and $\alpha^{+}$ belongs to the smallest eigenspace of $W^{+}.$
\end{lemma}
\begin{proof}
Since $\alpha^{+}\in H^{+}(M,\Bbb{R})$ is a non-degenerate harmonic 2-form, we have the following Weitzenb\"och formula
\begin{equation}
\label{[1.91]}
0=\langle \Delta\alpha^{+}, \alpha^{+} \rangle = \frac{1}{2} \Delta\mid  \alpha^{+} \mid ^2 + \mid   \nabla \alpha^{+}   \mid^2 + \langle(\frac{s}{3} - 2W^{+})\alpha^{+}, \alpha^{+} \rangle.
\end{equation} Moreover, taking into account that $w_1^{+}$ is the smallest eigenvalue of $W^{+}$ we infer $$\langle W^{+}(\alpha^{+}),\alpha^{+}\rangle\geq w_{1}^{+}\langle\alpha^{+},\alpha^{+}\rangle.$$ Hence, upon integrating (\ref{[1.91]}) over $M^4$ we use the above information to achieve

\begin{eqnarray}
\label{klh}
\int_{M}|\nabla\alpha^{+}|^{2} dV_{g}+\int_{M}\frac{s}{3}|\alpha^{+}|^{2}dV_{g}\geq 2\int_{M}\big(w_{1}^{+}+w_{1}^{-}\big)|\alpha^{+}|^{2}dV_{g}.
\end{eqnarray} Substituting (\ref{[1.6]}) into (\ref{klh}) we obtain
$$\int_{M}|\nabla \alpha^{+}|^{2}dV_{g}\geq \frac{2}{3}\int_{M}\big(6K_{1}^{\perp}-s\big)|\alpha^{+}|^{2}dV_{g},$$ which is the desired result.
\end{proof}

Before to proceed, we recall that $w_{1}^{+}: M^4\to (-\infty, 0]$ is a Lipschitz continuous function. Therefore, following the ideas developed in \cite{LeBrun} we define $f_{-}(x)=\min\{f(x),0\},$ where $f$ is a real-valued function on $M^4.$ From this, we have the following proposition, which is a slightly modification of Proposition 2.2 in  \cite{LeBrun}.

\begin{proposition}
\label{propT}
Let $M^4=X\sharp j \Bbb{CP}^2$ be a compact K\"ahler surface, where $X$ is the minimal model of $M^4.$ Suppose that $Y(M)<0.$ Then we have:
\begin{equation}
\label{67}
\frac{\pi^2}{2}\big(2\chi(M)+3\tau(M)\big)\le \Big(\int_{M}|\frac{1}{2}(K_{1}^{\perp}+\frac{s}{12})_{-}|^3 dV_{g}\Big)^{\frac{2}{3}} \vol(M,g)^{\frac{1}{3}}.
\end{equation} Moreover, if (\ref{67}) is actually an equality, then $M^4$ is half-conformally flat.
\end{proposition}

\begin{proof}
We start applying the same arguments used in the proof of Proposition 2.2 of \cite{LeBrun} in order to obtain
\begin{eqnarray}
\label{p10}
-\int_{M}\Big(\frac{2}{3}s+2w_{1}^{+}\Big)_{-}|\phi|^4 dV_{g}\ge \int_{M}(-s)|\phi|^{2}dV_{g}-4\int_{M}|\phi|^{2}|\nabla\phi|^{2}dV_{g},
\end{eqnarray} where $\phi$ is a solution of the Seiberg-Witten equations. Moreover, we already know that
\begin{eqnarray}
\label{p2}
\int_{M}(-s)|\phi|^4 dV_{g}\ge 4\int_{M}|\phi|^{2} |\nabla\phi|^{2} dV_{g}+\int_{M}|\phi|^{6} dV_{g},
\end{eqnarray} for details see Eq. (7) in \cite{LeBrun} (see also pg. 399 in \cite{scorpan}).

On the other hand, it follows immediately from (\ref{eigenvalues}) that
\begin{equation}
\label{p3}
\frac{2}{3}s+2w_{1}^{+}+2w_{1}^{-}\leq \frac{2}{3}s+2w_{1}^{+},
\end{equation} with equality if and only if $M^4$ is half-conformally flat. Hence, combining (\ref{p10}), (\ref{p2}), (\ref{[1.6]}) and (\ref{p3}) we get

\begin{eqnarray}
\label{g1}
-\int_{M}\big(\frac{s}{3}+4K_{1}^{\perp}\big)_{-} |\phi|^{4}dV_{g}\ge \int_{M}|\phi|^{6}dV_{g}.
\end{eqnarray} Now, we use the H\"older inequality to infer

\begin{equation}
\label{g2}
-\int_{M}\big(\frac{s}{3}+4K_{1}^{\perp}\big)_{-}|\phi|^{4} dV_{g}\leq \Big(\int_{M}|(\frac{s}{3}+4K_{1}^{\perp})_{-}|^{3}dV_{g}\Big)^{\frac{1}{3}}\Big(\int_{M}|\phi|^{6}dV_{g}\Big)^{\frac{2}{3}}
\end{equation}
and
\begin{equation}
\label{g3}
\int_{M}|\phi|^{6}dV_{g}\ge \vol(M,g)^{-\frac{1}{2}}\Big(\int_{M}|\phi|^{4}dV_{g}\Big)^{\frac{3}{2}}.
\end{equation} Therefore, comparing (\ref{g1}) with (\ref{g2}) we deduce

\begin{equation}
\int_{M}|\phi|^{6} dV_{g} \le \int_{M}|\big(\frac{s}{3}+4K_{1}^{\perp}\big)_{-}|^{3} dV_{g}.
\end{equation} We then use (\ref{g3}) to achieve

\begin{equation}
\label{g4}
\int_{M}|\phi|^4 dV_{g}\leq \Big(\int_{M}|\big(\frac{s}{3}+4K_{1}^{\perp}\big)_{-}|^{3} dV_{g}\Big)^{\frac{2}{3}} \vol(M,g)^{\frac{1}{3}}.
\end{equation} But, we already know from Seiberg-Witten theory (cf. Eq. (4) in \cite{LeBrun}) that
\begin{equation}
\label{g5}
32 \pi^{2}\big(2\chi(M)+3\tau(M)\big)\le\int_{M}|\phi|^{4}dV_{g}.
\end{equation} Whence, combining (\ref{g4}) with (\ref{g5}) we deduce (\ref{67}). Finally, if equality occurs in (\ref{67}) we have $w_{1}^{-}\equiv 0,$ and in this case we can use (\ref{eigenvalues}) to conclude that $M^4$ must be half-conformally flat, which finishes the proof of the proposition.
\end{proof}

\subsection{Additional Notation} We need to fix more notation that will be useful in the proofs of Theorems \ref{thmSub} and \ref{thmSub2}. To start with, we consider $M^n$ to be an $n$-dimensional compact submanifold in an $(n+m)$-dimensional Riemannian manifold $N^{n+m}.$ We adopt the following convention on the indices:

\begin{equation*}
1\leq i,\,j,\,k\leq n\,\,\hbox{and}\,\,\,n+1\leq\beta,\,\gamma\leq n+m.
\end{equation*} Moreover, $R_{ijkl},$ $\overline{R}_{ijkl},$ $\alpha$ and $\vec{H}$ stand for the Riemannian curvature tensor of $M^4,$ Riemannian curvature tensor of $N^{n+m},$ second fundamental form and mean curvature vector, respectively. From this it follows that

\begin{equation*}
R_{i\,j\,k\,l}=\overline{R}_{ijkl}+\sum_{\beta}\Big(\alpha_{ik}^{\beta}\alpha_{jl}^{\beta}-\alpha_{il}^{\beta}\alpha_{jk}^{\beta}\Big)
\end{equation*}
and
\begin{equation*}
R_{\beta\gamma\, k\,l}=\overline{R}_{\beta\gamma\, k\,l}+\sum_{i}\Big(\alpha_{ik}^{\beta}\alpha_{il}^{\gamma}-\alpha_{il}^{\beta}\alpha_{ik}^{\gamma}\Big).
\end{equation*} We further have $$\vec{H}=\frac{1}{n}\sum_{\beta,\,i}\alpha_{ii}^{\beta}e_{\beta}.$$ It is easy to check, from Gauss Equation, that
\begin{equation}
\label{rice}
Ric(e_{i})=\sum_{j}\overline{R}_{ijij}+\sum_{\beta,\,j}\Big[\alpha_{ii}^{\beta}\alpha_{jj}^{\beta}-\big(\alpha_{ij}^{\beta}\big)^{2}\Big].
\end{equation} Moreover, if $N$ has constant sectional curvature $c,$ then the scalar curvature of $M^n$ is given by

\begin{equation}
\label{scal}
s=n(n-1)c+n^{2}H^{2}-\|\alpha\|^{2},
\end{equation} where $H$ is the length of the mean curvature vector.

\section{Proof of the Main Results}
\subsection{Proof of Theorem \ref{thmT}}

\begin{proof}
First of all, for any function $\phi$ on $M^4$ we set $\overline{g}=e^{2\phi}g\in[g].$ From this it follows that $$\overline{s}=e^{-2\phi}(-6\Delta\phi-6|\nabla\phi|^2+s)$$  and $$6e^{2\phi}\overline{K}_{1}^\perp=6K_{1}^\perp-3\Delta\phi-3|\nabla\phi|^2,$$ which can be rewritten as $$e^{2\phi}\Big(\overline{K}_{1}^\perp+\frac{\overline{s}}{12}\Big)=K_{1}^{\perp}-\Delta\phi-|\nabla\phi|^{2}+\frac{s}{12}.$$

On the other hand, notice that $$s+3(w_{1}^{+}+w_{1}^{-})=\frac{1}{2}\Big(K_{1}^{\perp}+\frac{s}{12}\Big)$$ is a modified scalar curvature (cf. \cite{CDR} and \cite{Itoh}). Therefore, since $Y_{s,K_{1}^{\perp}}(M)\leq Y(M)<0,$ we can apply Itoh's theorem, Theorem A \cite{Itoh} to deduce that there is a metric $g\in[g]$ for which the modified scalar curvature $\frac{1}{2}\Big(K_{1}^{\perp}+\frac{s}{12}\Big)$ is a  negative constant. Hence, a straightforward computation combining these informations and Stokes formula yields
\begin{equation}
\label{p0}
\int_{M}|K_{1}^{\perp}+\frac{s}{12}|dV_{g}\leq \int_{M}e^{2\phi}|\overline{K}_{1}^{\perp}+\frac{\overline{s}}{12}|dV_{g}.
\end{equation}  Next, we use the Cauchy-Schwarz inequality as well as $dV_{\overline{g}}=e^{4\phi}dV_{g}$ to obtain
\begin{eqnarray}
\label{p1}
\frac{1}{\vol(M,g)}\Big(\int_{M}|\frac{1}{2}\Big(K_{1}^{\perp}+\frac{s}{12}\Big)|dV_{g}\Big)^{2}\leq \int_{M}|\frac{1}{2}\Big(\overline{K}_{1}^{\perp}+\frac{\overline{s}}{12}\Big)|^{2}dV_{\overline{g}},
\end{eqnarray} We then invoke Proposition \ref{propT} to infer

\begin{eqnarray}
\label{mn1}
\int_{M}|\frac{1}{2}\Big(\overline{K}_{1}^{\perp}+\frac{\overline{s}}{12}\Big)|^{2}dV_{\overline{g}}&\ge& [\vol(M,g)]^{\frac{1}{3}}\Big(\int_{M}|\frac{1}{2}\Big(K_{1}^{\perp}+\frac{s}{12}\Big)|^{3}dV_{g}\Big)^{\frac{2}{3}}\nonumber\\&\ge& \frac{\pi^{2}}{2}\big(2\chi(M)+3\tau(M)\big).
\end{eqnarray}

Since $Y(M)<0,$ we immediately deduce $\frac{1}{2}\Big(\overline{K}_{1}^{\perp}+\frac{\overline{s}}{12}\Big)\geq -1.$ From now on, without loss of generality, we may consider $g$ instead $\overline{g}.$ Thereby, we infer $$\inf_{g\in \mathcal{M}_{K_{1}^{\perp},s}}\int_{M}|\frac{1}{2}\Big(K_{1}^{\perp}+\frac{s}{12}\Big)|^2 dV_{g} \geq \inf_{g\in\mathcal{M}} \int_{M}|\frac{1}{2}\Big(K_{1}^{\perp}+\frac{s}{12}\Big)|^2 dV_{g},$$ where $\mathcal{M}_{K_{1}^{\perp},s}$ is the set of metrics such that $\big(K_{1}^{\perp}+\frac{s}{12}\big)\geq -1.$ Of which we obtain
\begin{equation*}
\vol_{K,s}(M^4)\geq\vol_{K_1^\perp,s}(M^4)\geq\inf_{g\in\mathcal{M}} \int_{M}|\frac{1}{2}\Big(K_{1}^{\perp}+\frac{s}{12}\Big)|^2 dV_{g}\geq\frac{9}{4}\vol_s(M^4).
\end{equation*} Finally, it suffices to repeat the final arguments of the proof of Lemma 2.6 in \cite{Itoh} (see also Proposition 3 in \cite{CDR}) to deduce

$$|Y_{K_{1}^{\perp},s}(M)|^{2}=\inf_{g\in\mathcal{M}} \int_{M}|\frac{1}{2}\Big(K_{1}^{\perp}+\frac{s}{12}\Big)|^2 dV_{g},$$ which finishes the proof of the theorem.
\end{proof}

\subsection{Proof of Theorem \ref{thmmono}}
\begin{proof}
The proof is inspired by the trend developed by LeBrun \cite{LeBrunP}. To start with, we invoke his Lemma 3.7  to deduce that, for any smooth positive function $f$ on $M^4,$ the rescaled Seiberg-Witten equations have a solution $(\phi,\,A),$ namely, 

\begin{eqnarray}
\left \{ \begin{array}{ll}
D^{A}\phi=0 \\
-iF_{A}=f\sigma(\phi).
                    \end{array} \right.
\end{eqnarray} Next, we remember that
$$0\geq \int_{M}\Big[4|\phi|^{2}|\nabla_{A}\phi|^{2}+s|\phi|^{4}+f|\phi|^{6}\Big]dV_{g}$$ (cf. Eq. (19) in \cite{LeBrunP}). Now, setting $\psi=2\sqrt{2}\sigma(\phi)$ and invoking Lemma \ref{lemM} we arrive at
\begin{eqnarray*}
0\geq \int_{M}\Big[s|\psi|^{2}+f|\psi|^{3}\Big]dV_{g}+\frac{2}{3}\int_{M}\Big(6K_{1}^{\perp}-s\Big)|\psi|^{2}dV_{g},
\end{eqnarray*} which can be written succinctly as
\begin{eqnarray*}
0\geq 4\int_{M}\Big[K_{1}^{\perp}+\frac{s}{12}\Big]|\psi|^{2}dV_{g}+\int_{M}f|\psi |^{3}dV_{g}.
\end{eqnarray*} We now consider $\gamma=\frac{1}{4}\psi$ to infer

\begin{eqnarray*}
-\int_{M}\Big(K_{1}^{\perp}+\frac{s}{12}\Big)|\gamma|^{2}dV_{g}\geq \int_{M}f|\gamma|^{3}dV_{g}.
\end{eqnarray*} From here it follows that
\begin{eqnarray*}
\int_{M}\Big[-\big(K_{1}^{\perp}+\frac{s}{12}\big)f^{-\frac{2}{3}}\Big]\Big[f^{\frac{2}{3}}|\gamma|^{2}\Big]dV_{g}\geq \int_{M}f|\gamma|^{3}dV_{g}.
\end{eqnarray*}By the H\"older inequality,  one verifies  that

\begin{eqnarray}
\label{12r}
\int_{M}|K_{1}^{\perp}+\frac{s}{12}|^{3}f^{-2}dV_{g}\geq \int_{M}f |\gamma|^{3}dV_{g}
\end{eqnarray}

On the other hand, using once more the H\"older inequality we ensure

\begin{eqnarray}
\Big(\int_{M}f^{4}dV_{g}\Big)^{\frac{1}{3}}\Big(\int_{M}f|\gamma|^{3}dV_{g}\Big)^{\frac{2}{3}}\geq \int_{M}f^{\frac{4}{3}}\big(f^{\frac{2}{3}}|\gamma|^{2}\big)dV_{g}.
\end{eqnarray} This combined with (\ref{12r}) yields

\begin{eqnarray}
\Big(\int_{M}f^{4}dV_{g}\Big)^{\frac{1}{3}}\Big(\int_{M}|K_{1}^{\perp}+\frac{s}{12}|^{3}f^{-2}\Big)^{\frac{2}{3}}\geq \int_{M}f^{2}|\gamma|^{2}dV_{g}.
\end{eqnarray} Note that $\gamma=\frac{\sqrt{2}}{2}\sigma(\phi)$ and thus $f\gamma=\frac{\sqrt{2}}{2}(-i F_{A}^{+}).$ Hence, from Lemma \ref{lemM}  and Proposition 4.5 in \cite{LeBrunP} we get

$$\Big(\int_{M}f^{4}dV_{g}\Big)^{\frac{1}{3}}\Big(\int_{M}|K_{1}^{\perp}+\frac{s}{12}|^{3}f^{-2}\Big)^{\frac{2}{3}}\geq 2\pi^{2}\beta^{2}(M),$$ for any smooth function $f$ on $M^4.$

Now, we choose a decreasing sequence of smooth positive functions $f_{k}$ on $M^4$ such that $$\lim_{k\to \infty}f_{k}^{2}=\frac{1}{2}\Big|K_{1}^{\perp}+\frac{s}{12}\Big|$$ uniformly  on $M^4.$ From this it follows that

$$\int_{M}\Big[\frac{1}{2}\big(K_{1}^{\perp}+\frac{s}{12}\big)\Big]^{2}dV_{g}\ge \frac{\pi^2}{4}\beta^{2}(M),$$ which was to be proved.

Finally, by assuming that $Y(M)<0,$ it suffices to invoke once more Proposition 3 in \cite{CDR} to conclude the proof of the theorem.
\end{proof}

\subsection{Proof of Theorem \ref{thmSub}}
\begin{proof}
To start with, for each point $p\in M^4,$ we consider $\{v_{1},v_{2},v_{3},v_{4}\}$ an orthonormal basis of $T_{p}M,$ and such that $\{\xi_1, \xi_2,...,\xi_m\}$ is an orthonormal referential in $\big(T_{p}M\big)^{\perp}.$ So, the Weingarten operator $A_{\xi_{\beta}},$ in the normal direction $\xi_{\beta},$ is given by $$\langle A_{\xi_{\beta}}v_{i},v_{j}\rangle=\langle \alpha(v_{i},v_{j}),\xi_{\beta}\rangle,$$ where $v_{i},v_{j}\in T_{p}M.$ From this, we have $$\vec{H}=\frac{1}{4}\sum_{\beta\geq 1}\big(tr\, A_{\xi_{\beta}}\big)\xi_{\beta}$$ and $H=\|\vec{H}\|.$ Moreover, we have $$\|\alpha\|^{2}=\sum_{\beta\geq 1}tr\,A_{\xi_{\beta}}^{2}.$$ We set $A_1 = A_{\xi_{1}}$ to be the Weingarten operator of the isometric immersion $f$ in the normal direction $\xi_{1} = \frac{1}{H}\vec{H} \in (T_{p}M)^\perp.$ In particular, notice that  $tr\, A_{1}=4H$ and $tr\,A_{\beta}=0$ for $\beta\ge 2.$ Furthermore, if $X,Y\in \big(T_{p}M\big)^{\perp},$ then $$\alpha(X,Y) = \sum_{\beta\geq 1} \langle A_{\beta}X,Y\rangle\xi_{\beta}.$$ Next, it follows from Gauss Equation that
$$ Ric(v_1) = \sum_{\beta\geq 1}\langle A_{\beta}v_1, v_1\rangle \sum_{i\neq 1}\langle A_{\beta}v_i,v_i\rangle  - \sum_{\beta\geq 1}\sum_{i\neq 1}\langle A_{\beta}v_i, v_1\rangle^2 + 3c.$$ Similarly, it is not difficult to check that
$$ Ric(v_2) = \sum_{\beta\geq 1}\langle A_{\beta}v_2, v_2\rangle \sum_{i\neq 2}\langle A_{\beta}v_i,v_i\rangle  - \sum_{\beta\geq 1}\sum_{i\neq 2}\langle A_{\beta}v_i, v_2\rangle^2 + 3c.$$

A straightforward computation taking into account these two above equations yields

\begin{eqnarray}
Ric(v_1) + Ric(v_2) &=& 4H\Big[\langle A_1v_1,v_1\rangle + \langle A_1v_2,v_2\rangle\Big]\nonumber\\&&-\sum_{\beta\geq 1}\Big[\langle A_{\beta}v_1,v_1\rangle^2 + \langle A_{\beta}v_2,v_2\rangle^2\Big]\nonumber\\&&-\sum_{\beta\geq 1}\Big[\sum_{i\neq 1}\langle A_{\beta}v_i,v_1\rangle^2 + \sum_{i\neq 2}\langle A_{\beta}v_i,v_2\rangle^2\Big] + 6c.
\end{eqnarray} Then, after some simple computations we arrive at

\begin{eqnarray}
Ric(v_1)+ Ric(v_2) &=& -\Big[\langle A_1v_1,v_1\rangle + \langle A_1v_2,v_2\rangle -2H\Big]^2\nonumber\\&& -\sum_{\beta\geq 1}\Big[\sum_{i\neq 1}\langle A_{\beta}v_i,v_1\rangle^2 + \sum_{i\neq 2}\langle A_{\beta}v_i,v_2\rangle^2\Big] + 4H^2+6c\nonumber\\&&+
 2\sum_{\beta\geq 1}\langle A_{\beta}v_1,v_1\rangle \langle A_{\beta}v_2,v_2\rangle
-\sum_{\beta\geq 2}\Big[\langle A_{\beta}v_1,v_1\rangle + \langle A_{\beta}v_2,v_2\rangle\Big]^2.
\end{eqnarray} At the same time, we invoke again Gauss Equation to infer $$K(v_1,v_2) = \sum_{\beta\geq 1}\langle A_{\beta}v_1,v_1\rangle \langle A_{\beta}v_2,v_2\rangle - \sum_{\beta\geq 1}\langle A_{\beta}v_1,v_2\rangle^2 + c.$$ Of which we deduce

$$Ric(v_1)+ Ric(v_2) \leq 2K(v_1,v_2) + 4c + 4H^2 + a + 2\sum_{\beta\geq 1}\langle A_{\beta}v_1,v_2\rangle^2,$$
where $a=\sum_{\beta\geq 1}\Big[\sum_{i\neq 1}\langle A_{\beta}v_i,v_1\rangle^2 + \sum_{i\neq 2}\langle A_{\beta}v_i,v_2\rangle^2\Big].$ In particular, we have
\begin{eqnarray}
 -\sum_{\beta\geq 1}\Big[\sum_{i\neq 1}\langle A_{\beta}v_i,v_1\rangle^2 &+& \sum_{i\neq 2}\langle A_{\beta}v_i,v_2\rangle^2\Big] + 2\sum_{\beta\geq 1}\langle A_{\beta}v_1,v_2\rangle^2 \nonumber\\&=& -\sum_{\beta\geq 1, i\neq 1,2}\Big[\langle A_{\beta}v_i,v_1\rangle^2 + \langle A_{\beta}v_i,v_2\rangle^2\Big] \leq 0.\nonumber
\end{eqnarray} This data allows us to infer $$2K(v_1,v_2) \geq Ric(v_1)+ Ric(v_2) -4H^2 - 4c$$ as well as
$$2K(v_3,v_4) \geq Ric(v_3)+ Ric(v_4) -4H^2 - 4c.$$ Hence, it follows that
$$4K^\perp(P) \geq s - 8H^2 - 8c,$$ where $s$ denotes the scalar curvature of $M^4$ and $P$ is the 2-plane generated by $v_1$ and $v_2.$ Therefore, we immediately have $$4K_1^\perp \geq s - 8H^2 - 8c$$ and this combined with (\ref{scal}) guarantees $$4K_{1}^\perp\geq -\|\alpha\|^2 + 8H^2 +4c,$$ as we wanted to prove.

Suppose that $\|\alpha\|^{2}<4\big(2H^{2}+c\big).$ From this, we apply Theorem \ref{thmAC} to conclude $H_{2}(M,\Bbb{Z})=0$ and then it suffices to use Lemma 2.2 in \cite{AC} to conclude that $M^4$ is homeomorphic to the sphere $\Bbb{S}^4.$ This concludes the proof of the theorem.
\end{proof}

\subsection{Proof of Theorem \ref{thmSub2}}
\begin{proof}
In order to prove the first assertion we assume that $M^4$ is isometrically immersed into  $\Bbb{S}^{4+m}.$ Moreover, we already know from Theorem \ref{thmSub} that $$4K_{1}^\perp\geq -\|\alpha\|^2 + 8H^2 +4.$$ Taking into account that $\|\alpha\|^{2}< 4$ we immediately deduce $4K_{1}^\perp> 8H^{2}\ge 0,$ in other words, $M^4$ has positive biorthogonal curvature. Hence, it is easy to see that (\ref{estAC}) holds (for $p=2$). Whence, we may use again Theorem \ref{thmAC} to deduce that $H_{2}(M,\Bbb{Z})=0$ and since $M^4$ has finite fundamental group, we can apply Lemma 2.2 in \cite{AC} to conclude that $M^4$ is homeomorphic to the sphere $\Bbb{S}^4,$ as desired.

Next, supposing $\|\alpha\|^{2}\leq 4,$ we deduce from Theorem \ref{thmSub} that $M^4$ has nonnegative biorthogonal curvature. In particular, if for every point of $M^4$ some biorthogonal curvature vanishes, then $H=0$ and $\|\alpha\|^{2}=4.$ Finally, if $m=1$ we invoke Chern-Do Carmo-Kobayashi theorem \cite{carmo} to conclude that $M^4$ must be  $\Bbb{S}_{c_{1}}^{2}\times \Bbb{S}_{c_{2}}^{2}.$ So, the proof is completed.
\end{proof}

\section{Appendix}
In this appendix we going to provide topological obstructions for the existence of Einstein structures by using the modified Yamabe invariant $Y_{1}^{\perp}(M)$ defined in (\ref{Yi}). To this end, we need the following lemma. 

\begin{lemma}\label{thObstEinstein}
Let $M^4$ be a 4-dimensional oriented compact manifold admitting an Einstein metric $g.$ Suppose that $Y_1^\perp(M)\leq0.$ Then the following assertions hold:
\begin{enumerate}
\item  $$\chi(M)\geq\frac{1}{576\pi^2}|Y_1^\perp(M)|^2.$$ Moreover, if the equality holds, then $M^4$ is either flat or $\Bbb{H}^2_c\times\Bbb{H}^2_c /\Gamma.$
\item If $g$ is half-conformally flat, then $$\chi(M)\geq \frac{1}{384\pi^{2}}|Y_1^\perp(M)|^2.$$ In particular, if the equality holds, then $M^4$ is a compact complex-hyperbolic 4-manifold $\Bbb{C}\mathcal{H}^{2}/\Gamma.$
\end{enumerate}
\end{lemma}

\begin{proof}
We start recalling that
\begin{equation}
\label{eqT}
|w_{1}^{\pm}|^2\le \frac{2}{3}|W^{\pm}|^2.
\end{equation} Moreover, the equality holds in (\ref{eqT}) if and only if $w_{3}^{\pm}=w_{2}^{\pm}$ (cf. Lemma 3.2 (a) in \cite{noronha2}). From here it follows that

\begin{equation}
\label{plh}
| W^+ |^2 + |W^-|^2 \geq\frac{3}{2}\Big[(w_1^+)^2+(w_1^-)^2\Big]=\frac{1}{24}\Big[(6w_1^+)^2+(6w_1^-)^2\Big].
\end{equation}

On the other hand, since $M^4$ admits an Einstein metric $g$ we may use Chern-Gauss-Bonnet formula (\ref{characteristic}) jointly with (\ref{plh}) to get
\begin{equation*}
8\pi^2\chi(M)\geq\frac{1}{24}\int_M\left(s^{2}+(6w_{1}^{+})^2+(6w_{1}^{-})^2\right)dV_g.
\end{equation*} Next, by the Cauchy-Schwarz inequality and (\ref{[1.6]}) we infer
\begin{eqnarray*}
8\pi^2\chi(M)&\geq&\frac{1}{72}\int_M(s+6w_1^++6w_1^-)^2dV_g=\frac{1}{72}\int_M 144(K_1^\perp)^2dV_g,
\end{eqnarray*} so that

\begin{equation}
\label{kjg}
4\pi^{2}\chi(M)\geq \int_{M}|K_{1}^{\perp}|^{2}dV_{g}.
\end{equation} Moreover, if the equality holds in (\ref{kjg}), then $w_1^+=w_1^-=\frac{s}{6}$ and $w_2^\pm=w_3^\pm.$ In this case $M^4$ must be $\Bbb{H}^2_c\times\Bbb{H}^2_c /\Gamma.$

At the same time, taking into account that $$4\pi^2\chi(M)\geq\inf_{g\in\mathcal{M}}\int_M|K_{1}^{\perp}|^2dV_g,$$ we may use Proposition 3 in \cite{CDR} (see also Lemma 2.6 in \cite{Itoh}) to deduce $$4\pi^2\chi(M)\geq\frac{1}{144}|Y_1^\perp(M)|^2.$$ In particular, the equality holds if and only if $M^4$ is either flat or $\Bbb{H}^2_c\times\Bbb{H}^2_c /\Gamma,$ which establishes the fist assertion.

Next, the proof of the second assertion looks like that one of the previous assertion. Indeed, without loss of generality we may assume that $M^4$ is self-dual. In such a case, a standard computation allows us to obtain

\begin{eqnarray}
8\pi^2 \chi(M)\geq \frac{1}{24}\int_{M}\Big[(6w_{1}^{+})^{2}+s^{2}\Big]dV_{g}\geq \frac{1}{48}\int_{M}\Big[6w_{1}^{+}+s\Big]^{2}dV_{g}\geq 3\inf_{g\in\mathcal{M}}\int_M |K_{1}^{\perp}|^2dV_g.
\end{eqnarray} To conclude it suffices to use again Proposition 3 in \cite{CDR} to achieve $$\chi(M)\geq \frac{1}{384\pi^{2}}|Y_1^\perp(M)|^2,$$ which gives the requested result.
\end{proof}

For what follows, it is important to recall that Taubes \cite{taubes} showed that for any smooth compact oriented four-dimensional manifold $X,$ there is an integer $j$ such that $M^4=X\sharp j \overline{\Bbb{CP}}^2$ admits half-conformally flat metrics. The minimal number for $j$ is called {\it Taubes invariant}, which is unknown for most of $4$-manifolds. 

Next, as an application of Lemma \ref{thObstEinstein} we deduce the following obstruction result.

\begin{proposition}\label{corObstEinstein}
$\Bbb{H}^2_c\times\Bbb{H}^2_c /\Gamma\sharp j(\Bbb{S}^1\times\Bbb{S}^3)$ does not admit Einstein metric provided that $j>\frac{4}{9}\chi( \Bbb{H}^2_c\times\Bbb{H}^2_c /\Gamma).$
\end{proposition}

\begin{proof}
First, we consider $M=\Bbb{H}^2_c\times\Bbb{H}^2_c /\Gamma\sharp j(\Bbb{S}^1\times\Bbb{S}^3)$ and $e=\chi(\Bbb{H}^2_c\times\Bbb{H}^2_c /\Gamma).$ Hence, we deduce $\chi(M)=e-2j$ and $\tau(\Bbb{H}^2_c\times\Bbb{H}^2_c /\Gamma)=0.$ Further, since $\Bbb{H}^2_c\times\Bbb{H}^2_c /\Gamma$ is the minimal model of $M^4$ we may invoke Theorem 3.9 of \cite{LeBrunTrends} to infer
\begin{equation*}
Y(\Bbb{H}^2_c\times\Bbb{H}^2_c /\Gamma)=-8\pi\sqrt{e}.
\end{equation*} We then use Proposition 3 in \cite{Petean} to deduce $Y(M)=Y(\Bbb{H}^2_c\times\Bbb{H}^2_c /\Gamma)=-8\pi\sqrt{e}.$

We now suppose that $M^4$ admits an Einstein metric. Thereby, since $Y_1^\perp(M)\leq Y(M)<0,$ we can apply Lemma \ref{thObstEinstein} to infer
\begin{equation*}
576\pi^2\chi(M)\geq|Y_1^\perp(M)|^2\geq|Y(M)|^2.
\end{equation*} From this it follows that $$576\pi^2(e-2j)\geq64e\pi^2,$$ so that $j\leq\frac{4}{9}e,$ which is a contradiction. So, the proof is completed.
\end{proof}

\begin{acknowledgement}
The authors want to thank the referees for their careful readings, relevant remarks and valuable suggestions. The second named author was partially supported by grants from CNPq/Brazil (Grant: 303091/2015-0) and PRONEX-FUNCAP/CNPq/Brazil. He also would like to thank the Department of Mathematics - Universidade Federal da Bahia, where part of this work was carried out, for the warm hospitality.
\end{acknowledgement}

\end{document}